%% file: BoturPasekaSmolkajan.tex
\documentclass[sn-mathphys-ay]{sn-jnl}

\usepackage{algorithm}%
\usepackage{algorithmicx}%
\usepackage{algpseudocode}%
\usepackage{amsmath, amsfonts, amssymb, amsthm, bm, stmaryrd, enumerate, mathtools}
\usepackage[title]{appendix}%
\usepackage{booktabs}%
\usepackage{comment}
\usepackage{enumitem}
\usepackage{graphicx} 
\usepackage{listings}%
\usepackage{multirow}%
\usepackage{mathrsfs}%
\usepackage{manyfoot}%
\usepackage{textcomp}%
\usepackage{tikz}
\usepackage{tikz-cd}
\usepackage{times}
\usepackage{xcolor}%
\usepackage{xspace}

\tikzset{commutative diagrams/.cd}
\tikzcdset{row sep/tiny=0.12 em}

\usetikzlibrary{intersections}

\newlength\myheight

        \usetikzlibrary{calc}

\theoremstyle{thmstyleone}%

\newtheorem{Def}{Definition}[section]
\newenvironment{definition}{\begin{Def} \rm}{\end{Def}}
\newtheorem{lemma}[Def]{Lemma}
\newtheorem{proposition}[Def]{Proposition}
\newtheorem{corollary}[Def]{Corollary}
\newtheorem{theorem}[Def]{Theorem}
\newtheorem{example}[Def]{Example}
\newtheorem{remark}[Def]{Remark}

\newtheorem{claim}{Claim}[Def]

\newcommand{\C}{{\mbox{${\mathcal C}$}}}
\DeclareMathOperator{\kernel}{ker}
\DeclareMathOperator{\image}{im}

\newcommand{\setin}[3]{\{#1\in#2\;|\;#3\}}

\newcommand{\downset}{\mathop{\downarrow}\!}

\newcommand{\lin}[1]{\langle #1\rangle}
\newcommand{\conjun}{\mathrel{\wedge}}
\newcommand{\disjun}{\mathrel{\vee}}

\newcommand{\komma}{,\hspace{0.3em}}
\newcommand{\adj}{^\star}

\newcommand{\after}{\mathrel{\circ}}

\newcommand{\Cat}[1]{\ensuremath{\mathbf{#1}}}

\newcommand{\op}{\ensuremath{^{\mathrm{op}}}}
\newcommand{\idmap}[1][]{\ensuremath{\mathrm{id}_{#1}}}

\renewcommand{\Im}{\ensuremath{\mathrm{Im}}}

\newcommand{\coker}{\ensuremath{\mathrm{coker}}}

\newcommand{\powerset}{\mathcal{P}}
\newcommand{\nul}{\ensuremath{\underline{0}}}

\newcommand{\EndoHom}[1]{\mathcal{E}{\kern-.5ex}\textit{n}{\kern-.2ex}\textit{d}{\kern-.2ex}\textit{o}(#1)}

\begin{document}

\title[A dagger kernel category of complete OMLs]{A dagger kernel category of complete orthomodular lattices}


\author[1]{\fnm{Michal} \sur{Botur}}\email{michal.botur@upol.cz}
\equalcont{These authors contributed equally to this work.}

\author*[2]{\fnm{Jan} \sur{Paseka}}\email{paseka@math.muni.cz}
\equalcont{These authors contributed equally to this work.}

\author[2]{\fnm{Richard} \sur{Smolka}}\email{394121@mail.muni.cz}
\equalcont{These authors contributed equally to this work.}

\affil[1]{\orgdiv{Department of Algebra and Geometry}, \orgname{Palack\'y University Olomouc}, \orgaddress{\street{17.\ listopadu 12}, \city{771 46 Olomouc}, 
\country{Czech Republic}}}

\affil*[2]{\orgdiv{Department of Mathematics and Statistics}, \orgname{Masaryk University}, \orgaddress{\street{Kotl\'a\v rsk\'a 2}, \city{611\,37 Brno}, 
\country{Czech Republic}}}



\abstract{Dagger kernel categories, a powerful framework for studying quantum phenomena within category theory, provide a rich mathematical structure that naturally encodes key aspects of quantum logic. This paper focuses on the category $\Cat{SupOMLatLin}$ of complete orthomodular lattices with linear maps. We demonstrate that $\Cat{SupOMLatLin}$ itself forms a dagger kernel category, equipped with additional structure such as dagger biproducts and free objects. A key result establishes that every morphism in $\Cat{SupOMLatLin}$ admits an essentially unique factorization as a zero-epi followed by a dagger monomorphism. This factorization theorem, along with the dagger kernel category structure of $\Cat{SupOMLatLin}$, provides new insights into the interplay between complete orthomodular lattices and the foundational concepts of quantum theory.
}




\maketitle

\section{Introduction}\label{sec1}


Dagger kernel categories, introduced in~\cite{HeJa}, provide a powerful framework for studying quantum phenomena within the elegant setting of category theory. 

At their core, dagger categories are equipped with a special involution, denoted by $\dagger$, on morphisms. This involution, reminiscent of the adjoint operation in Hilbert spaces, allows for the abstract representation of physical processes and their time-reversals. The concept of kernels, familiar from algebra, is extended to these categories. Kernels, in this context, capture the notion of subobjects or ''subsystems" within a given object. 

The interplay between the dagger structure and the existence of kernels leads to a rich mathematical structure that naturally encodes key aspects of quantum logic. This includes the possibility of representing propositions about quantum systems as subobjects and exploring their logical relationships.
Dagger kernel categories offer a concise and abstract framework for studying various aspects of quantum theory, including quantum computation, quantum information, and the foundations of quantum mechanics. 
This paper extends the investigation of dagger kernel categories, with a particular focus on their connections to complete orthomodular lattices.

We build upon the work of Jacobs \cite{Jac}, who introduced the fundamental category $\Cat{OMLatGal}$ of orthomodular lattices with Galois connections and explored its properties. Our work also draws inspiration from \cite{BLP}, where the category $\Cat{OMLatLin}$ of orthomodular lattices with linear maps was studied. 

It is worth noting that the category $\Cat{OMLatGal}$, which plays a pivotal role in Jacobs' work, was first defined by Crown in~\cite{Crown75}. However, Crown did not conduct a systematic investigation of this category, nor did he explore its significance within the context of categorical quantum logic, particularly in relation to dagger kernel categories.

 The main results of this paper are as follows.

 We introduce a special category, denoted by $\Cat{SupOMLatLin}$, whose objects are complete orthomodular lattices and whose morphisms are linear maps between them. We establish the following key properties of $\Cat{SupOMLatLin}$:

\begin{itemize}
    \item $\Cat{SupOMLatLin}$ is a dagger kernel category equipped with additional structure, including dagger biproducts and free objects.
    \item Every morphism $f\colon X \to Y$ in $\Cat{SupOMLatLin}$ admits an essentially unique factorization 
     as a zero-epi morphism followed by a dagger monomorphism, where the factorization occurs  through the principal ideal $\downarrow f(1)$ of $X$.
\end{itemize}

The paper is organized as follows. Section~\ref{sec2} introduces the fundamental category $\Cat{SupOMLatLin}$ of orthomodular lattices with linear maps between them and investigates some of its properties. Subsequently, Section~\ref{sec:dkccoml} recalls the definition of a dagger kernel category and proves that $\Cat{SupOMLatLin}$ yields a dagger kernel category. In Section~\ref{sec:factor}, we will prove that every morphism in $\Cat{SupOMLatLin}$ admits an essentially unique factorization as a zero-epi followed by a kernel. Section~\ref{sec:biproducts} investigates the existence of arbitrary biproducts and free objects in $\Cat{SupOMLatLin}$. The paper ends with some final remarks in Section~\ref{ConclusionSec}.

This work presumes that readers are already acquainted with the fundamental concepts and results related to dagger categories and orthomodular lattices. For those interested in exploring these topics further, it is recommended to refer to the works by \cite{HeJa,Jac}, and \cite{Kalmbach83}.

\section{A dagger category of complete orthomodular lattices
}\label{sec2}

The following definition lays the foundation for understanding ortholattices, algebraic structures that generalize Boolean algebras. A key feature of an ortholattice is the presence of an orthocomplement operation, which generalizes the notion of negation in Boolean logic. The definition also introduces the concept of orthomodularity, a crucial property that distinguishes orthomodular ortholattices from more general ortholattices. This property has significant implications for the structure and behavior of these lattices, particularly in the context of quantum logic and other applications.

\medskip

\begin{definition}\label{OMLatDef}{ 
A meet semi-lattice $(X,\conjun, 1)$ is called an {\em ortholattice} if it
comes equipped with a function $(-)^{\perp}\colon X \to X$ satisfying:
\begin{itemize}
   \item $x^{\perp\perp} = x$;
   \item $x \leq y$ implies $y^\perp \leq x^\perp$;
   \item $x \conjun x^\perp = 1^\perp$.
\end{itemize}

\noindent One can then define a bottom element as $0 = 1 \conjun
1^{\perp} = 1^\perp$ and join by $x\disjun y = (x^{\perp}\conjun
y^{\perp})^{\perp}$, satisfying $x\disjun x^{\perp} = 1$.

We write $x\perp y$ if and only if $x\leq y^{\perp}$. 

Such an ortholattice is called {\em orthomodular} if it satisfies (one of)
the three equivalent conditions:
\begin{itemize}
\item $x \leq y$ implies $y = x \disjun (x^\perp \conjun y)$;

\item $x \leq y$ implies $x = y \conjun (y^\perp \disjun x)$;

\item $x \leq y$ and $x^{\perp} \conjun y = 0$ implies $x=y$.
\end{itemize}}
\end{definition}

\medskip

In what follows we recall the definition a dagger category, a fundamental structure in the study of categorical quantum mechanics and other areas of theoretical computer science. A dagger category is a category equipped with an involutive functor, known as the dagger, which provides a notion of adjoint or dual for morphisms. This structure allows for the representation and study of various physical and computational phenomena, such as quantum processes and reversible computations.

\medskip

\begin{definition}
\label{DagcatDef}
A {\it dagger} on a category \C\ is a functor ${}^\star \colon \C\op \to \C$ that is involutive and the identity on objects. A category equipped with a dagger is called a {\it dagger category}.

\[
\begin{tikzcd}[row sep=1cm, column sep=.75cm, ampersand replacement=\&]
	A\arrow[r, bend left,"{g^{*}}"]  \&  B	\arrow[r, bend left,"{f^{*}}"] %
	\arrow[l,  "g"]\&C \arrow[l,  "f"]\&%
	A\arrow[rr, bend left,"{f^{*}\circ g^{*}=(g\circ f)^{*}}"]  \& \& C  \arrow[ll,  "g\circ f"]\&%
	A\arrow[rr, bend left,"{1_A=(1_A)^{*}}"]  \& \& A  \arrow[ll,  "1_A"]
\end{tikzcd}
\]

Let \C\ be a dagger category. A morphism $f \colon A \to B$ is called a {\it dagger monomorphism} if $f^{\star} \circ f={\idmap}_A$, and $f$ is called a {\it dagger isomorphism} if 
$f^{\star} \circ f = {\idmap}_A$ and $f \circ f^\star = {\idmap}_B$. A {\it dagger automorphism} is a dagger isomorphism $f \colon A \to A$.
\end{definition}

\medskip 

In a dagger category, limits and colimits are dual concepts. Applying the dagger operation ${}\adj$ to a limit cone results in a colimit cone, and vice versa.

\medskip 

The following example provides a fundamental illustration of a dagger category that finds widespread applications in various fields, including linear algebra, quantum mechanics, and quantum information theory. The category of finite-dimensional vector spaces with linear transformations as morphisms, equipped with the dagger operation, serves as a foundational example for understanding the concepts of dagger monomorphisms, dagger isomorphisms, and other key properties of dagger categories in more abstract settings.

\medskip 

\begin{example} \rm 
Consider a category ${\mathcal C}$ with the following properties:
\begin{itemize}
\item The objects of ${\mathcal C}$ are denoted as ${\mathbf K}^{n}$, where 
${\mathbf K}$ can be either the real numbers ($\mathbb R$) or the complex numbers ($\mathbb C$) and $n$ is a positive integer.
\item The morphisms in ${\mathcal C}$ are represented by matrices of size $m\times n$ over ${\mathbf K}$, where $m$ and $n$ are positive integers.
\item Composition of morphisms is defined by matrix multiplication when the dimensions are compatible.
\item There exists an involution operation ($^*$) defined on morphisms:
 \begin{itemize}
   \item For ${\mathbf K}=\mathbb R$, the involution is transpose ($A^{*}$ = $A^{T}$).
 \item For ${\mathbf K}=\mathbb C$ , the involution is conjugate-transpose ($A^{*}$ = ${\overline A}^{T}$).
\end{itemize}
\end{itemize}
With these properties, the category ${\mathcal C}$  can be classified as a dagger category.
 \end{example}

\medskip

We now present a novel construction that organizes complete orthomodular lattices (as described in \cite{BLP} for orthomodular lattices) into a dagger category with linear maps as morphisms. A linear map between orthomodular lattices is characterized by the existence of an adjoint function satisfying a specific orthogonality condition. This framework provides a foundation for studying the algebraic and categorical properties of these structures.

\medskip

\begin{definition}\label{def:SupOMLatLin}
The category \Cat{SupOMLatLin} has complete orthomodular
lattices as objects.
A morphism $f \colon X\rightarrow Y$ in \Cat{SupOMLatLin} is a 
function $f \colon X\rightarrow Y$ between the underlying sets such that 
there is a function $h \colon Y \to X$ and, 
for any $x \in X$ and $y \in Y$,
\[ f(x) \perp y \text{ if and only if } x \perp h(y). \]
We say that $h$ is an {\it adjoint} of a {\em linear map} $f$. 
It is clear that adjointness is a symmetric property: if a map $f$ possesses an adjoint $h$, then $f$ is also an adjoint of $h$. 
We denote $\Cat{Lin}(X,Y)$ the set of all linear maps from $X$ to $Y$.
If $X=Y$ we put $\Cat{Lin}(X)=\Cat{Lin}(X,X)$.

Moreover, a map $f \colon X \to X$ is called {\it self-adjoint} if $f$ is an adjoint of itself.

The identity morphism on $X$ is the self-adjoint identity map $\idmap \colon X\rightarrow X$. Composition of $\smash{X \stackrel{f}{\rightarrow} Y
  \stackrel{g}{\rightarrow} Z}$ is given by usual composition of maps.
\end{definition}

\medskip

We immediately see that \Cat{SupOMLatLin} is really a category. Namely, if $h$ is an adjoint 
of $f$ and $k$ is an adjoint of $g$ we have, for any $x \in X$ and $z \in Z$,
\begin{align*}
    g(f(x)) \perp z \text{ if and only if } f(x) \perp k(z)  \text{ if and only if }  
    x \perp h(k(z)).
\end{align*}
Hence $h\circ k$ is an adjoint of $g\circ f$. Moreover, for any $x, y \in X$, 
\begin{align*}
    \idmap(x) \perp y \text{ if and only if } x \perp y  \text{ if and only if }  
    x \perp \idmap(y)
\end{align*}
and $\idmap \colon X\rightarrow X$ is self-adjoint.

Linear maps between complete orthomodular lattices exhibit an order-preserving property, as demonstrated in \cite[Lemma 5]{BLP}. It is worth noting that linear maps were investigated in a more general setting in \cite{PaVe4}.

\medskip

Our guiding example is as follows. It provides a specific instance of objects of the category of complete orthomodular lattices with linear maps. It demonstrates that the set of closed subspaces of a Hilbert space, equipped with the operations of intersection and orthogonal complementation, forms a complete orthomodular lattice. Furthermore, it shows how bounded linear operators between Hilbert spaces naturally induce linear maps between the corresponding lattices of closed subspaces, with the adjoint of the induced map being determined by the adjoint of the original operator.

\medskip

\begin{example}\label{exam:Hilbert}
Let $H$ be a Hilbert space. We denote the closed subspace spanned by a subset $S \subseteq H$ by $\lin S$. Let $C(H) = \{ \lin S \colon S \subseteq H \}$. Then $C(H)$ is a complete orthomodular lattice 
such that $\wedge=\cap$ and $P^{\perp}$is the orthogonal complement of a closed subspace $P$
of $H$.

Let $f \colon H_1 \to H_2$ be a bounded linear map between Hilbert spaces 
and let $f\adj$ be the usual adjoint of $f$. Then the induced map 
$C(H_1) \to C(H_2) \komma \lin S \mapsto \lin{f(S)}$ has the adjoint 
$C(H_2) \to C(H_1) \komma \lin T \mapsto \lin{f\adj(T)}$.
\end{example}

\medskip

The following lemma establishes a fundamental equivalence between different characterizations of linear maps between complete orthomodular lattices, connecting the existence of right order-adjoints, the existence of adjoints in the sense of the previous definition, and the preservation of arbitrary joins.

\medskip

\begin{lemma} \label{lem:lattice-adjoint}
Let $f \colon X \to Y$ and $h \colon Y \to X$ be maps between 
complete orthomodular
lattices. Then the following conditions are equivalent:
\begin{itemize}

\item[\rm (i)] $f$ is a linear map and possesses the adjoint $h \colon Y \to X$. 

\item[\rm (ii)] $f$ possesses a right order-adjoint $\hat{h} \colon Y \to X$ such that 
$\hat{h}={}^{\perp}\circ h\circ {}^{\perp}$.

\item[\rm (iii)] $f$ preserves arbitrary joins.
\end{itemize}
\end{lemma}
\begin{proof} Ad (i) $\Leftrightarrow$ (ii): See \cite[Lemma 6]{BLP}.

The equivalence between (ii) and (iii) is a well-established result.
%
\end{proof}

\medskip

If $f$ does have an adjoint, it has exactly one adjoint and we shall denote it by $f\adj$.

\medskip 

The following theorem demonstrates that the category of complete orthomodular lattices with linear maps forms a dagger category. 

\medskip 

\begin{theorem}\label{OMLisdagger}
    \Cat{SupOMLatLin} is  a dagger category.
\end{theorem}
\begin{proof}
    Since every morphism in \Cat{SupOMLatLin} has a unique adjoint we obtain that 
    ${}\adj\colon \Cat{SupOMLatLin}\op \to \Cat{SupOMLatLin}$ is involutive and the identity on objects.
\end{proof}

\begin{remark}\rm 
Jacobs in \cite{Jac} introduced the category \Cat{OMSupGal} with complete orthomodular
lattices as objects.
A morphism $X\rightarrow Y$ in \Cat{OMSupGal} is a pair $f =
  (f_{\bullet}, f^{\bullet})$ of ``antitone'' functions $f_{\bullet}\colon X\op
  \rightarrow Y$ and $f^{\bullet}\colon Y \rightarrow X\op$ forming a Galois
  connection (or order adjunction $f^{\bullet}\dashv f_{\bullet}$): 
  \begin{center}
      $x\leq f^{\bullet}(y)$\quad iff\quad $y\leq f_{\bullet}(x)$\quad for\ $x\in X$\ and\ $y\in Y$.
  \end{center}
  
\medskip
The identity morphism on $X$ is the pair $(\bot,\bot)$ given by the
self-adjoint map $\idmap^{\bullet} = \idmap[\bullet] = (-)^{\perp} \colon X\op
\rightarrow X$. Composition of $\smash{X \stackrel{f}{\rightarrow} Y
  \stackrel{g}{\rightarrow} Z}$ is given by:
$$\begin{array}{rclcrcl}
(g\after f)_{\bullet}
& = &
g_{\bullet} \after \bot \after f_{\bullet}
& \quad\mbox{and}\quad &
(g\after f)^{\bullet}
& = &
f^{\bullet} \after \bot \after g^{\bullet}.
\end{array}$$

The Galois connection yields that $f_{\bullet}$ preserves meets, as right
adjoint, and thus sends joins in $X$ (meets in $X\op$) to meets in
$Y$, and dually, $f^{\bullet}$ preserves joins and sends joins in $Y$ to
meets in $X$. 

The category $\Cat{OMSupGal}$ has a dagger, namely by twisting:
$$\begin{array}{rcl}
(f_{\bullet},f^{\bullet})^{\adj}
& = &
(f^{\bullet},f_{\bullet}).
\end{array}$$
\end{remark}

\begin{definition}\label{def:daggeriso}
Two dagger categories, $\mathbf{C}$ and $\mathbf{D}$, are said to be {\em{dagger isomorphic}} if there exists a functor $F: \mathbf{C} \to \mathbf{D}$ that satisfies the following conditions:

\begin{enumerate}
    \item \textbf{Isomorphism:} $F$ is an isomorphism of categories, meaning it is both full, faithful, and essentially surjective. 

    \item \textbf{Dagger Preservation:} $F$ preserves the dagger operation, meaning that for any morphism $f$ in $\mathbf{C}$, we have $F(f^{\adj}) = F(f)^{\adj}$.
\end{enumerate}

\end{definition}

\medskip

This theorem establishes a connection between categories 
\Cat{OMSupGal} and \Cat{SupOMLatLin} by showing they are essentially the same from a categorical perspective, being related by a dagger isomorphism.  This isomorphism preserves not only the categorical structure but also respects the dagger operation, revealing a deep symmetry in these mathematical frameworks.

\medskip

 \begin{theorem} \label{thm:daggiso}
 \Cat{OMSupGal} and \Cat{SupOMLatLin} are dagger isomorphic via functors 
 $\Lambda\colon \Cat{SupOMLatLin} \to \Cat{OMLatGal}$ and $\Gamma\colon  \Cat{OMLatGal} \to \Cat{OMLatLin} $ which are identities on objects and otherwise given by 
 $$
 \Lambda(f)=(\bot \after f, \bot \after f^{\adj})\quad\text{ and }
\quad\Gamma (f_{\bullet}, f^{\bullet})=\bot \after f_{\bullet}.
 $$
 \end{theorem}
 \begin{proof} First, let us check that $ \Lambda$ and $\Gamma$ are correctly defined. Suppose that $X, Y$ are complete orthomodular 
 lattices and $f \colon X \to Y$ a linear map. Evidently, 
 $\bot \after f$ and $\bot \after f^{\adj}$ are antitone maps. 
 Let $x\in X$\ and\ $y\in Y$. We compute:
\begin{align*}
  x\leq (\bot \after f^{\adj})(y)\quad 
      &\text{iff}\quad 
      x\leq  (f^{\adj}(y))^{\bot}\quad 
      \text{iff}\quad 
      x\perp  f^{\adj}(y)\quad 
      \text{iff}\quad 
      f(x)\perp  y\quad\\ 
      &\text{iff}\quad y\leq (f(x))^{\bot}\quad 
      \text{iff}\quad y\leq (\bot \after f)(x).
  \end{align*}

Conversely, let $(f^{\bullet},f_{\bullet})$ be a 
morphism  in \Cat{OMSupGal}\ between $X$ and $Y$. Again, let 
$x\in X$\ and\ $y\in Y$. We have:
\begin{align*}
  x\perp (f^{\bullet}(y))^{\perp}\quad 
      &\text{iff}\quad 
      x\leq  f^{\bullet}(y)\quad 
      \text{iff}\quad  y\leq  f_{\bullet}(x) %
      \quad  \text{iff}\quad  y\perp  (f_{\bullet}(x))^{\perp}.
  \end{align*}

Let us check that $\Lambda$ is both full and faithful. Assume that 
$X, Y$ are complete orthomodular lattices 
and $(f^{\bullet},f_{\bullet})$ be a 
morphism  in \Cat{OMSupGal}\ between $X$ and $Y$.

Then $\Gamma (f_{\bullet}, f^{\bullet})=\bot \after f_{\bullet}:X\to Y$ is a linear map between $X$ and $Y$ and 
\begin{align*}\Lambda(\bot \after f_{\bullet})=%
(\bot \after (\bot \after f_{\bullet}), \bot \after (\bot \after f_{\bullet})^{\adj})=(f_{\bullet}, \bot \after (\bot \after f^{\bullet}))=
(f_{\bullet}, f^{\bullet}).
\end{align*}

Now, let $f,g\colon X\to Y$ be linear maps such that 
$\Lambda(f)=\Lambda(g)$. Then $\bot \after f=\bot \after g$. 
Hence $f=\bot \after\bot \after f=\bot \after\bot \after g=g$. 

Since $\Lambda$ is identity on objects it is 
essentially surjective.

Let $f \colon X \to Y$ be a linear map. We compute: 
\begin{align*}\Lambda(f\adj)=%
(\bot \after f\adj, \bot \after {f\adj}\adj)%
=(\bot \after f\adj, \bot \after f)=(\bot \after f,\bot \after f{\adj})^{*}=\Lambda(f)\adj
\end{align*}
\end{proof}

\begin{remark}\label{rem:omlatlin}\rm
The aforementioned Theorem would allow us to formulate most of the following statements without proofs. However, we prefer to describe, for example, the specific realization of dagger monomorphisms, dagger 
epimorphisms, etc. The reason for this is that we are motivated by Example \ref{exam:Hilbert}.
\end{remark}

\section{A dagger kernel category of complete orthomodular lattices}
\label{sec:dkccoml}

Let us recall the following elementary results and definitions.

\medskip 

Lemma \ref{DownsetLem} (as mentioned e.g. in \cite[Lemma 3.4]{Jac}) demonstrates a well-known result that for any element $a$ in an orthomodular lattice $X$, the set of all elements less than or equal to $a$ (the principal downset of $a$) inherits the structure of an orthomodular lattice with respect to the induced order and a suitably defined orthocomplement.

\medskip 

\begin{lemma} {\rm\cite[Lemma 3.4]{Jac}}
\label{DownsetLem}
Let $X$ be an orthomodular lattice, with element $a\in X$. 
The (principal) downset $\downset a = \setin{u}{X}{u \leq a}$ is
  again an orthomodular lattice, with order, meets and
  joins as in $X$, but with its own orthocomplement $\perp_a$ given
  by $u^{\perp_a} = a \conjun u^{\perp}$, where $\perp$ is the
  orthocomplement from $X$.
\end{lemma}

\medskip 

The Sasaki projection plays a fundamental role in the theory of orthomodular lattices, capturing a unique type of projection operation that respects the quantum-like structure of these mathematical objects. The following definition introduces a map that will prove to be essential for understanding how elements in an orthomodular lattice relate to each other, particularly in contexts where classical Boolean logic needs to be generalized. The projection is named after Shôichirô Sasaki, who made significant contributions to the study of quantum logic and its algebraic foundations.

\medskip

\begin{definition}\label{def:Sasaki projection}
		Let $X$ be an orthomodular lattice. Then the map $\pi_a:X\to X$, $y\mapsto a\wedge(a^\perp\vee y)$ is called the \emph{Sasaki projection} to $a\in X$.
	\end{definition}

    \medskip 

We need the following facts about Sasaki projections 
 (see \cite{foulis1962note} and \cite{LiVe}):

 \medskip 
	
	\begin{lemma}\label{lem:Sasaki projection facts}
		Let $X$ be an orthomodular lattice, and let $a\in X$. Then for each $y,z\in L$ we have
		\begin{itemize}
			\item[(a)] $y\leq a$ if and only if $\pi_a(y)=y$;
			\item[(b)] $\pi_a(\pi_a(y^\perp)^\perp))=a\wedge y\leq y$;
			\item[(c)] $\pi_a(y)=0$ if and only if $y\leq a^\perp$;
			\item[(d)] $\pi_a(y)\perp z$ if and only if $y\perp \pi_a(z)$.
		\end{itemize}
	\end{lemma}    

\medskip

The following definition recalls two fundamental concepts in the study of morphisms between complete orthomodular lattices: the kernel and range. While these notions mirror their counterparts in linear algebra and other mathematical structures, their behavior in the context of orthomodular lattices reveals unique properties that reflect the quantum-logical nature of these spaces. The kernel comprises all elements that map to zero, while the range consists of all elements that can be reached by the morphism. A key distinction between them is that while a kernel always forms a complete orthomodular lattice, the range only forms a complete lattice.

\medskip

\begin{definition}\label{def:kernel}
    Let $f \colon X \to Y$ be a morphism of complete orthomodular lattices. 
    We define the {\it kernel} and the {\it range} of $f$, respectively, by
\begin{align*}
\kernel f \;=\; & \{ x \in X \colon f(x) = 0 \}, \\
\image f \;=\; & \{ f(x) \colon x \in X \}.
\end{align*}
\end{definition}

\medskip 

Corollary \ref{Sasself} reveals several important structural properties of the Sasaki projection that make it particularly useful in the study of orthomodular lattices. The fact that it is both self-adjoint and idempotent shows it behaves like a classical projection operator, while its image being the down-set of the element $a$ provides a clear geometric picture of where the projection maps elements. These properties are reminiscent of orthogonal projections in Hilbert spaces, which is not surprising given the deep connections between orthomodular lattices and quantum mechanics.

\medskip 

    \begin{corollary}\label{Sasself}
     Let $X$ be a complete orthomodular lattice, and let $a\in X$. Then 
     $\pi_a$ is self-adjoint, idempotent and 
     $\image \pi_a=\downset a$.
 \end{corollary}

\medskip 

 The following characterization of Sasaki maps reformulates 
\cite[page 144]{Kalmbach83}. 

\medskip 
 \begin{proposition}\label{charsasa}
Let $X$ be a complete orthomodular lattice, $f\colon X\to X$ a map and $a=f(1)$. 
Then the following conditions are equivalent:
\begin{itemize}

\item[\rm (i)] $f=\pi_a$;
\item[\rm (ii)] $f$ is self-adjoint linear, idempotent and $\image f=\downset a$;
\item[\rm (iii)] $f$ is self-adjoint linear, $f(a)=a$ and $x\in\downset a$ implies $f(x)\leq x$;
\item[\rm (iv)] $f$ is self-adjoint linear and $x\in\downset a$ implies $f(x)=x$.
\end{itemize}
\end{proposition}
\begin{proof} Ad (i)$\Rightarrow$ (ii): This is exactly Corollary \ref{Sasself}.

 Ad (ii)$\Rightarrow$ (iii): Let $x\in\downset a$. Then $x=f(u)$  for some $u\in X$. Since $f$ 
 is idempotent we conclude $x=f(u)=f(f(u))=f(x)$ . 

 Ad (iii)$\Rightarrow$ (iv): Let $x\in\downset a$. Then $f(x)\leq x$ 
 and $f(x^{\perp_a})\leq x^{\perp_a}$. From 
 Lemma \ref{DownsetLem} we know that $\downset a$ is an orthomodular lattice 
 and $x\vee  x^{\perp_a}=a$. Therefore also $a=f(a)=f(x\vee  x^{\perp_a})=%
 f(x)\vee  f(x^{\perp_a})\leq  f(x)\vee  x^{\perp_a}\leq a$. We conclude that 
 $f(x)^{\perp_a}\wedge  x=0$. Since $\downset a$ is an orthomodular lattice 
 we obtain $f(x)=x$.

 Ad (iv)$\Rightarrow$ (i): Assume (iv). The condition $x\in\downset a$ implies $f(x)=x$ is exactly a reformulated condition (1) 
 in \cite[page 144]{Kalmbach83}, 
 condition (2) in \cite[Page 144]{Kalmbach83} means exactly that $f$ is self-adjoint linear. Moreover condition (3) \cite[page 144]{Kalmbach83}  follows from condition (2) using Lemma \ref{lem:lattice-adjoint}. Since any map $f$  
 satisfying (1), (2) and (3) is identical to 
 $\pi_a$ (see \cite[pages 144--145]{Kalmbach83}) we finished the proof.
\end{proof}

\medskip 
Lemma \ref{cores} establishes an elegant relationship between restrictions and corestrictions of linear maps between orthomodular lattices, focusing specifically on how these operations interact with the adjoint operation. 
\medskip 

\begin{lemma}\label{cores}
Let $f$ be a linear map between complete orthomodular lattices $X$ and $Y$, $y\in Y$ and $f(1)\leq y$. Then 
the corestriction $f|^{\downset y} \colon X \to \downset y$ of $f$ and the restriction 
$f^{*}|_{\downset y} \colon {\downset y} \to X$ of $f^{*}$ are linear maps such that 
$(f|^{\downset y})^{*}=f^{*}|_{\downset y}$.
\end{lemma}
\begin{proof} Evidently, both maps are correctly defined. Let us show that $(f|^{\downset y})^{*}=f^{*}|_{\downset y}$.
 Let $x \in X$ and $u\in \downset y$.  
 We compute: 
\begin{align*}
f|^{\downset y}(x) \perp_y u&\text{ if and only if } f(x)\leq u^{\perp_y}=y\wedge u^{\perp} 
\text{ if and only if } f(x)\leq u^{\perp}\\
&\text{ if and only if } f^{*}(u)\leq x^{\perp} \text{ if and only if } f^{*}|_{\downset y}(u) \leq x^{\perp}\\
&\text{ if and only if } x\perp f^{*}|_{\downset y}(u).
\end{align*}
\end{proof}

\medskip

Similarly as in \cite[Lemma 3.4]{Jac} we have the following lemma. 
This lemma reveals a fundamental connection between down-sets in orthomodular lattices and dagger monomorphisms, showing how the Sasaki projection naturally arises as the composition of a certain embedding and its adjoint. 

\medskip

\begin{lemma} \label{DownsetLemma}
Let $X$ be a complete orthomodular lattice, with element $a\in X$. 
There is a dagger monomorphism $\downset a \rightarrowtail X$ in
  \Cat{SupOMLatLin}, for which we also write $a$, with
$$\begin{array}{rclcrcl}
a(u) & = & u
& \quad\mbox{and}\quad &
a^{*}(x) & = & \pi_a(x).
\end{array}$$
Moreover, $a\circ a^{*}=\pi_a$.
\end{lemma}
\begin{proof} Let $x \in X$ and $u\in \downset a$.  

We compute: 
\begin{align*}
    a(u) \perp x &\text{ if and only if } \pi_a(u) = u \perp x  \text{ if and only if }  
    u \perp \pi_a(x)\\
    &\text{ if and only if } u \leq \pi_a(x)^{\perp} %
    \text{ if and only if } u \leq \pi_a(x)^{\perp} \wedge a= \pi_a(x)^{\perp_a}.
\end{align*}
The first equivalence follows from Lemma \ref{lem:Sasaki projection facts}(a), 
the second one from Lemma \ref{lem:Sasaki projection facts}(d). The remaining equivalences are evident. 

From Lemma \ref{lem:Sasaki projection facts}(a) we also obtain that the 
 map $a\colon \downset a\rightarrow X$ is a dagger monomorphism since:
 \begin{align*}
a^{*}(a(u)) = \pi_a(u)=u.
 \end{align*}
 The statement $a\circ a^{*}=\pi_a$ is a consequence of Lemma \ref{lem:Sasaki projection facts}(a) and the property that $\pi_a$ is idempotent.
\end{proof}

\medskip

The following lemma provides a beautiful characterization of the kernel of a morphism between complete orthomodular lattices, expressing it elegantly in terms of a down-set of an orthogonal element. What makes this result particularly interesting is that it not only identifies the kernel's structure but also confirms that this kernel inherits the complete orthomodular lattice structure from its parent complete orthomodular lattice. 

\medskip

\begin{lemma} \label{KernelLemma}
Let $f \colon X \to Y$ be a morphism of complete orthomodular lattices. Then 
$\kernel f=\downset f^{*}(1)^{\perp}$ is a complete orthomodular lattice.
\end{lemma}
\begin{proof} Let $x \in X$. We compute: 
\begin{align*}
x\in \kernel f&\text{ if and only if } f(x)=0  \text{ if and only if } 
f(x)\perp 1\\ 
&\text{ if and only if } x\perp f^{*}(1) \text{ if and only if } 
x\leq f^{*}(1)^{\perp}.
\end{align*}
\end{proof}

\medskip
Corollary \ref{SSasself} reveals two important properties of self-adjoint morphisms on complete orthomodular lattices. The first part extends our understanding of kernels by showing how they relate to down-sets of orthocomplements when the morphism is self-adjoint, while the second part establishes a subtle order-theoretic inequality involving double orthocomplements. These properties are particularly significant as they demonstrate how self-adjointness interacts with the orthomodular structure.
\medskip

\begin{corollary}\label{SSasself}
     Let $X$ be a complete orthomodular lattice, and let 
     $f \colon X \to X$ be a self-adjoint morphism of complete   orthomodular lattices. Then 
$\kernel f=\downset f(1)^{\perp}$ and 
$f(f(y^\perp)^\perp))\leq y$ for all $y\in X$.
 \end{corollary}
  \begin{proof} It is enough to check that 
 $f(f(y^\perp)^\perp))\leq y$. We compute: 
 \begin{align*} 
 f(f(y^\perp)^\perp))\leq y&%
  \text{ if and only if }  
  f(f(y^\perp)^\perp))\perp y^\perp \\
   &\text{ if and only if }   
   f(y^\perp)^\perp\perp f(y^\perp).
 \end{align*}
 \end{proof}

 \medskip

 We show that \Cat{SupOMLatLin} has a {\it zero object} $\nul$; this means that there is, for any complete 
orthomodular lattice $X$, a unique morphism $\nul \to X$ and hence 
also a unique morphism $X \to \nul$. 

The zero object $\nul$ will be 
one-element orthomodular lattice $\{0\}$.  Let us show that $\nul$ is indeed an 
initial object in $\Cat{SupOMLatLin}$. Let $X$ be an arbitrary complete  orthomodular
lattice. The only function $f \colon \nul \to X$ is $f(0)=0$. 
Since we may identify $\nul$ with $\downset 0$ we have that $f$ is is a dagger monomorphism and 
it has an adjoint $f^* \colon X\to \nul$ defined by $f^*(x)=\pi_0(x)=0$.

For objects $X$ and $Y$, we denote by $0_{X,Y}=X\rightarrow \nul \rightarrow Y$ the morphism uniquely factoring through $\nul$.

\medskip

The following definition introduces the concepts of kernels, cokernels, and their dagger counterparts, and lays the groundwork for understanding and working with various algebraic structures  within the framework of category theory. The concepts of a zero–mono morphism and  of 
a zero–epi morphism  characterize morphisms based on their interaction with the zero object and the zero morphism. They provide a way to distinguish between "injective-like" and "surjective-like" morphisms in the categorical context.

\medskip

\begin{definition}\label{weakkernel}
\begin{enumerate}
\item[]
\item For a morphism $f \colon A \to B$ 
in arbitrary category with zero morphisms, we say that a morphism  
$k\colon K \to A$ is a {\em  kernel} of $f$ if 
$fk=0_{K,B}$, and if $m\colon M\to A$ satisfies $fm=0_{M,B}$  then  there is a unique morphism 
$u \colon M \to K$ such that $ku = m$. 

We sometimes write $\kernel f$ for $k$ or $K$.

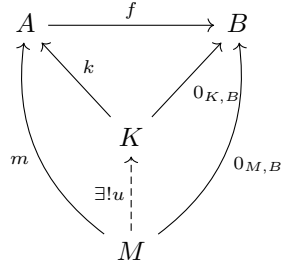
\[
\begin{tikzcd}[row sep=1cm, column sep=.75cm, ampersand replacement=\&]
	A\arrow[rr, "f"] \& \& B\\
	\& K\arrow[ul, "k"']\arrow[ur, "0_{K,B}"']\\
	\& M\arrow[uul, bend left, "m"]\arrow[u,dashed,"\exists!u"]\arrow[uur, bend right, "0_{M,B}"']
\end{tikzcd}
\]

\item A {\em cokernel} in a category with zero morphisms is a kernel in the opposite category.

 \item   For a morphism $f \colon A \to B$ 
in arbitrary dagger category with zero morphisms, we say that a morphism  
$k\colon  K \to A$ is a {\em weak dagger kernel} of $f$ if 
$fk=0_{K,B}$, and if $m\colon M\to A$ satisfies $fm=0_{M,B}$  then $kk^{*}m=m$.

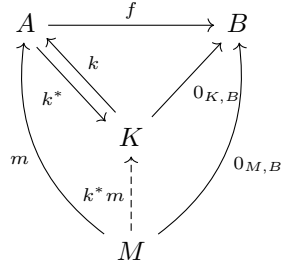
\[
\begin{tikzcd}[row sep=1cm, column sep=.75cm, ampersand replacement=\&]
	A\arrow[rr, "f"] \arrow[dr, shift right, swap, "k^{*}"]\& \& B\\
	\& K\arrow[ul,shift right, "k"']\arrow[ur, "0_{K,B}"']\\
	\& M\arrow[uul, bend left, "m"]\arrow[u,dashed,"k^{*}m"]\arrow[uur, bend right, "0_{M,B}"']
\end{tikzcd}
\]

A {\em weak dagger kernel category} is a dagger category with zero morphisms where every morphism has a weak dagger kernel.

\item A {\em dagger kernel category} is a dagger category with a zero object, hence zero morphisms, where each morphism $f$ has a weak dagger kernel $k$ (called {\em dagger kernel}) 
that additionally satisfies $k^{*}k=1_K$.

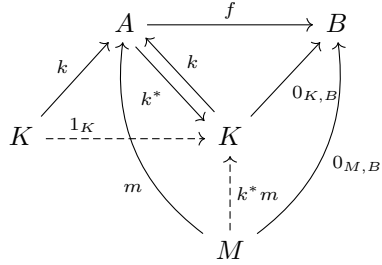
\[
\begin{tikzcd}[row sep=1cm, column sep=.75cm, ampersand replacement=\&]
\&	A\arrow[rr, "f"] \arrow[dr, shift right, swap, "k^{*}"]\& \& B\\
K\arrow[ur, swap,"k"']\arrow[rr,dashed, "1_K\phantom{xxxxxxx}"]\&	\& K\arrow[ul,shift right, "k"']\arrow[ur, "0_{K,B}"']\\
\&	\& M\arrow[uul, bend left, "m", pos=0.35, labels=below left]\arrow[u,dashed, swap,"k^{*}m"]\arrow[uur, bend right, "0_{M,B}"']
\end{tikzcd}
\]

\item A {\em (weak) dagger cokernel} in a dagger category with zero morphisms is a (weak) dagger  kernel in the opposite category.

\item  A morphism $f \colon A \to B$   in arbitrary dagger 
kernel category with a zero object  is said to be {\em zero–epi} 
if, for $g \colon B \to C$,  
$gf = 0$ implies $g = 0$. A morphism $f$ is said to be {\em zero–mono} if $f^{*}$ is zero–epi.
\end{enumerate}
\end{definition}

\medskip

By Lemma \ref{KernelLemma}, the kernel of a morphism $f \colon X \to Y$ in $\Cat{SupOMLatLin}$, as defined for orthomodular lattices, coincides with the kernel of f as defined in Definition \ref{weakkernel}.

Furthermore, every dagger kernel is inherently a kernel. In the context of $\Cat{SupOMLatLin}$, we will demonstrate that dagger kernels always exist and are, in fact, equivalent to kernels.

Theorem \ref{OMLatLinDagKerCatThm} hence establishes a fundamental property of the category $\Cat{SupOMLatLin}$, demonstrating its crucial structure as a dagger kernel category.

\medskip

\begin{theorem}
\label{OMLatLinDagKerCatThm}
  The category $\Cat{SupOMLatLin}$ is a dagger kernel category. The
  dagger kernel of a morphism $f \colon X \to Y$ is 
  $k \colon\!  \downset k \to X$, where $k = f^*(1)^{\perp} \in X$, like in
  Lemma~\ref{KernelLemma}.
\end{theorem}
\begin{proof} Since $\downset k=\kernel f$ 
the composition $f \after k$ is the zero map $0_{\kernel f,Y}=\downset k \rightarrow Y$. 

Assume now that $m\colon M\to X$ satisfies $fm=0_{M,Y}$ and $z\in M$. Then 
$m(z)\in \kernel f$, i.e., $m(z)\leq {f^*(1)}^{\perp}$ and $m(z)$ commutes with ${f^*(1)}^{\perp}$.

We compute:
\begin{align*}
kk^{*}m(z)&= \pi_{f^*(1)^{\perp}}(m(z))=f^*(1)^{\perp} \wedge (f^*(1) \disjun m(z))\\
&=f^*(1)^{\perp} \wedge  m(z)=m(z).
\end{align*}

From Lemma \ref{DownsetLemma} we know that $k \colon\!  \downset k \to X$ is 
a dagger monomorphism, i.e., $k^{*}k=1_{\kernel f}$.
\end{proof}
\section{Factorization in \Cat{SupOMLatLin} }\label{sec:factor}

In this section, we will demonstrate that every morphism in \Cat{SupOMLatLin} admits an essentially unique factorization as a zero-epi followed by a kernel.

Following the approach in \cite{Jac}, we provide explicit descriptions of cokernels and zero-epis. It is important to note that in a dagger category equipped with dagger kernels, dagger cokernels also exist. These cokernels can be chosen to be dagger epi and are defined as $\coker(f) = \ker(f^{*})^{*}$.

\medskip

\begin{lemma}
\label{CokerLem}
The cokernel of a map $f\colon X\rightarrow Y$ in \Cat{SupOMLatLin} is
$$\coker(f) = \pi_{f(1)^{\perp}}|^{\downset f(1)^{\perp}}\colon Y \to  \downset f(1)^{\perp}.$$

\noindent Then:
\begin{align*}
\mbox{$f$ is zero-epi}
\;\;&\smash{\stackrel{\textrm{def}}{\Longleftrightarrow}}\;\;
\coker(f) = 0
\;\Longleftrightarrow\;
f(1)=1,\\
\mbox{$f$ is zero-mono}
\;\;&\smash{\stackrel{\textrm{def}}{\Longleftrightarrow}}\;\;
\coker(f^{*}) = 0
\;\Longleftrightarrow\;
f^{*}(1)=1.
\end{align*}
\end{lemma}
\begin{proof} It follows immediately from the preceding 
remark, Lemma \ref{DownsetLemma} and Theorem 
\ref{OMLatLinDagKerCatThm}. Namely, 
\begin{align*}
\coker(f)= \ker(f^{*})^{*}=(f(1)^{\perp})^{*}=%
\pi_{f(1)^{\perp}}|^{\downset f(1)^{\perp}}.
\end{align*}
The remaining part follows from \cite[Lemma 4]{HeJa} and 
the the description of a cokernel.
\end{proof}

\begin{definition}\cite{HeJa}
    For a morphism $f\colon X \to Y$ in a dagger kernel category \Cat{D}, its {\em image} is defined as $\ker(\coker(f))$. This results in:
\begin{enumerate}[label=(\roman*), ref=(\roman*)]
\item A representative morphism $i_f$ (which can be chosen to be a dagger mono).
\item An object $\Im(f)$ (the domain of $i_f$).
\end{enumerate}

Both $i_f$ and $\Im(f)$ are called the {\em image} of $f$, and they are unique up to isomorphism of the domain.
\end{definition}

\medskip

\begin{remark}\label{rem:factor}\rm 
Explicitly, the image and an essentially unique zero-epi and kernel factorization 
of $f$ can be obtained as follows \cite{HeJa}: 

\begin{enumerate}
\item Take the kernel $k$ of $f^{\adj}$: 

\begin{tikzcd}
\ker(f^{\adj}) \arrow[r, hook, "k"] & Y \arrow[r, "f^{\adj}"] & X
\end{tikzcd}
\item Then define $i_f$ as the kernel of $k^{\adj}$, as illustrated in the following diagram:
\begin{equation}
  \label{ImageEqn}
\begin{tikzcd}
& \Im(f) = \ker(k^{\adj})\arrow[r, hook, "i_f"] & Y \arrow[r, twoheadrightarrow, "k^{\adj}"] & \ker(f^{\adj}) \\
& X \arrow[u, "e_f", dashed] \arrow[ur, "f"'] & & &
\end{tikzcd}
\end{equation}
\item The uniquely determined morphism $e_f\colon X\to \Im f$ is zero-epi 
and $e_f=(i_f)^{\adj}\circ f$. 
\item This zero-epi and kernel factorization $f = i_{f} \after e_{f}$ from~(\ref{ImageEqn}) is unique up-to isomorphism.
\item Let $m_f=e_f\after i_{f^{\adj}}$. Then we can factorise $f$ as: \begin{equation}
\label{ImageCoimageEqn}
\begin{tikzcd}[column sep=large]
X \arrow[rr, two heads, "\textrm{coimage}"', %
"{(i_{f^{\adj}})^{\adj}}"] %
\arrow[rrrrr, "f"', bend left=30] 
& & \Im(f^{\adj}) \arrow[r, two heads, tail, "m_f", "{\begin{array}{c}\textrm{zero-epi} \\ \textrm{zero-mono}\end{array}}"']  
& \Im(f) \arrow[rr, hook, "i_f", "\textrm{image}"']
& & Y
\end{tikzcd}
\end{equation}
\end{enumerate}
\end{remark}

We briefly examine how the factorization from Remark~\ref{rem:factor} applies specifically to the category \Cat{SupOMLatLin}. Note that these results are analogous to those for Hilbert spaces \cite[Example 3]{HeJa}.

\medskip

\begin{proposition}
\label{ImFacLem}
For a map $f\colon X\rightarrow Y$ in \Cat{SupOMLatLin} one has:

\medskip

\begin{center}
\begin{tikzcd}[row sep=small, column sep=small]
\Im(f) = \downset f(1) \arrow[r, hook] 
& Y&\text{with} 
& i_{f}=f(1),\\
X \arrow[r, "e_f", two heads] & \downset f(1)& \text{is} 
& e_f=f|^{\downset f(1)},\\
\Im(f^{\adj}) \arrow[r, two heads, tail, "m_f"]&\Im(f^{\adj})& \text{is} 
& m_f=f|^{\downset f(1)}_{{\downset {f^{\adj}}(1)}}.\\
\end{tikzcd}
\end{center}
\end{proposition}
\begin{proof} This can be verified by a direct application of the definitions. We compute:
\begin{align*}
    \Im(f) &= \ker(\coker(f)) = \ker(\pi_{f(1)^{\perp}}|^{\downset f(1)^{\perp}})=%
    \left((\pi_{f(1)^{\perp}}|^{\downset f(1)^{\perp}})^{*}(f(1)^{\perp})%
    \right)^{\perp}\\
    &=f(1)^{\perp\perp}=f(1).
\end{align*}
  Since   $e_f=(i_f)^{\adj}\circ f$ and $m_f=e_f\after i_{f^{\adj}}$ we conclude that 
  $e_f=f|^{\downset f(1)}$ and 
  $m_f=f|^{\downset f(1)}_{{\downset {f^{\adj}}(1)}}$. 
\end{proof}

As a result,  in 
\Cat{SupOMLatLin}, factorization (\ref{ImageEqn}) rewrites as: 

\begin{equation}
  \label{OMLImageEqn}
\begin{tikzcd}
& \downset f(1)\arrow[r, hook, "f(1)"] & Y \\
& X \arrow[u, "f|^{\downset f(1)}", dashed] \arrow[ur, "f"'] & 
\end{tikzcd}
\end{equation}

and factorization (\ref{ImageCoimageEqn}) rewrites as: 

\begin{equation}
\label{OMLImageCoimageEqn}
\begin{tikzcd}[column sep=large]
X \arrow[rr, two heads, "\textrm{dagger epi}"', %
"{({f^{\adj}}(1))^{\adj}}"] %
\arrow[rrrrr, "f"', bend left=30] 
& & \downset {f^{\adj}}(1)\arrow[r, two heads, tail, 
"f|^{\downset f(1)}_{{\downset {f^{\adj}}(1)}}", "{\begin{array}{c}\textrm{zero-epi} \\ \textrm{zero-mono}\end{array}}"']  
& \downset f(1) \arrow[rr, hook, "f(1)", "\textrm{dagger mono}"']
& & Y
\end{tikzcd}
\end{equation}

\begin{corollary}\label{cor:imrIm}
   Let $f\colon X\rightarrow Y$ in \Cat{SupOMLatLin} satisfy 
   $\image(f) = \downset f(1)$. 
   Then the unique factorization (\ref{ImageEqn}) rewrites as: 

\begin{equation}
  \label{corOMLImageEqn}
\begin{tikzcd}
& \downset f(1)\arrow[r, hook, "id_{Y}|_{\downset f(1)}"] & Y \\
& X \arrow[u, "f|^{\downset f(1)}", dashed] \arrow[ur, "f"'] & 
\end{tikzcd}
\end{equation}
\end{corollary}

\begin{corollary}\label{cor:SasimrIm}
   Let $\pi_a\colon X\rightarrow X$ be a Sasaki projection. 
   Then the unique factorization (\ref{ImageEqn}) rewrites as: 

\begin{equation}
  \label{OMLImageEqn2}
\begin{tikzcd}
& \downset a\arrow[r, hook, "\pi_{a}|_{\downset a}"] & X \\
& X \arrow[u, "\pi_{a}|^{\downset a}", dashed] \arrow[ur, "\pi_a"'] & 
\end{tikzcd}
\end{equation}
Moreover, $\pi_{a}|^{\downset a}$ is dagger epi, $\pi_{a}|_{\downset a}$ is dagger mono, 
$\pi_{a}|^{\downset a}=\left(\pi_{a}|_{\downset a}\right)^{*}$, and 
$\pi_{a}|^{\downset a}\circ \pi_{a}|_{\downset a}=id_{\downset a}$.
\end{corollary}

\section{Biproducts and free objects in \Cat{SupOMLatLin} }\label{sec:biproducts}

The category $\Cat{SupOMLatLin}$ of complete orthomodular lattices with linear maps exhibits a rich algebraic structure. One key aspect of this richness is the presence of arbitrary dagger biproducts. 
The dagger structure, essential for capturing the involutive nature of the category, endows the biproduct with a self-duality.

By a {\it dagger biproduct} of  objects $A, B$ in a dagger 
 category $\C$ with a zero object, we mean a coproduct $\begin{tikzcd} A \arrow[r, "\iota_A"] & A \oplus B & B \arrow[l, "\iota_B"'] \end{tikzcd}$ such that $\iota_A, \iota_B$ are dagger monomorphisms and ${\iota_B}^\star \circ \iota_A = 0_{A,B}$. The dagger biproduct of an arbitrary set of objects is defined in the expected way.

\medskip

 \begin{proposition}
\label{BiprodProp}
The category $\Cat{SupOMLatLin}$ has arbitrary dagger biproducts $\bigoplus$.
Explicitly, $\bigoplus_{i\in I} X_i$ is the  cartesian product of  orthomodular lattices $X_i$, $i\in I $, $I$ an index set.

The coprojections $\kappa_j\colon X_j
\to \bigoplus_{i\in I} X_i$  are defined by 
\begin{align*}
  (\kappa_j)(x) = x_{j=}&\quad\text{ with } \quad
x_{j=}(i)=\begin{cases}
	x&\text{if}\ \ i=j; \\
	0&\text{otherwise.}
\end{cases}
\end{align*}
and
$(\kappa_j)^*((x_i)_{i\in I}) = x_j$. The dual product structure is given by
$p_j = (\kappa_j)^*$.
\end{proposition}
\begin{proof}
  Let us first verify that $\kappa_j$ is a well-defined morphism of
  $\Cat{SupOMLatLin}$. 
  
  We compute: 
  \begin{align*}
   \kappa_j(x)\perp (y_i)_{i\in I}&
   \text{ if and only if } 
   x\leq y_j^{\perp} 
   \text{ if and only if }  
   y_j \leq x^{\perp} \\
   &\text{ if and only if } 
   (\kappa_j)^*((y_i)_{i\in I})\perp x.
  \end{align*}

\noindent Also, $\kappa_j$ is a dagger 
monomorphism since:
$$(\kappa_{j})^{*}\big((\kappa_{j})(x)\big) 
= (\kappa_{j})^{*}\big(x_{j=}\big) 
= x.$$ 

\noindent For $i \neq j$, we show  that $(\kappa_j)^*\after \kappa_i$ is the zero morphism. Namely, 

\begin{align*}
\big((\kappa_{j})^{*} \after \kappa_{i}\big)_{*}(y)
& = (\kappa_{j})^{*}\big((\kappa_{i})(y)\big)= %
(\kappa_{j})^{*}\big(x_{i=}\big)=0. 
\end{align*}

It is evident that 
$(x_i)_{i\in I}\in \bigoplus_{i\in I} X_i$ 
if and only if 
$(x_i)_{i\in I}=\bigvee_{j\in I} (x_j)_{j=}$.

Let us show that $\bigoplus_{i\in I} X_i$ is indeed a coproduct. 

Suppose
that morphisms $f_i \colon X_i \to Y$ in $\Cat{SupOMLatLin}$ are given.  We then define the
map $\bigoplus_{i\in I} f_i \colon \bigoplus_{i\in I} X_i \to Y$ by
$(\bigoplus_{i\in I} f_i)\big((x_i)_{i\in I}\big) =%
\bigvee_{j\in I} f_j(x_j)$
and the map  
$(\bigoplus_{i\in I} f_i)^{*} \colon Y \to \bigoplus_{i\in I} X_i $ by
$(\bigoplus_{i\in I} f_i)^{*}(y) = (f_i^*(y))_{i\in I}$. We compute:

\begin{align*}
&(\bigoplus_{i\in I} f_i)\big((x_i)_{i\in I}\big) 
\perp y\ 
   \text{ if and only if } 
   \bigvee_{i\in I} f_i(x_i) \leq y^{\perp} 
    \\
    \text{ if and only if } &f_i(x_i) \leq y^{\perp} \text{ for all } 
    {i\in I}
    \text{ if and only if } 
    f_i(x_i) \perp y \text{ for all } 
    {i\in I}\\ 
    \text{ if and only if } &
    x_i\perp f_i^*(y)  \text{ for all } 
    {i\in I}
     \text{ if and only if } 
     f_i^*(y) \leq x_i^{\perp} \text{ for all } 
    {i\in I}\\ 
    \text{ if and only if }&  
    (\bigoplus_{i\in I} f_i)^{*}(y) \leq 
    (x_i^{\perp})_{i\in I}
    \text{ if and only if }  
    (\bigoplus_{i\in I} f_i)^{*}(y) \perp 
    (x_i)_{i\in I}.
\end{align*}

Also, 
\begin{align*}
\big((\bigoplus_{i\in I} f_i)\circ \kappa_{j}\big)(x)&=
(\bigoplus_{i\in I} f_i)(x_{j=})=0\vee f_j(x)=f_j(x).
\end{align*}

\noindent Moreover,
if $g\colon  \bigoplus_{i\in I} X_i \rightarrow Y$ also satisfies
$g\after\kappa_{i} = f_{i}$, then: 
\begin{align*}
g\big((x_i)_{i\in I}\big)&=
g(\bigvee_{j\in I} (x_j)_{j=})=%
\bigvee_{j\in I}g((x_j)_{j=})=%
\bigvee_{j\in I}f_j(x_j)=%
(\bigoplus_{i\in I} f_i)\big((x_i)_{i\in I}\big). 
\end{align*}
\end{proof}

\medskip

In the category $\Cat{SupOMLatLin}$ of complete orthomodular lattices with linear maps, a fundamental concept is that of a free object. Intuitively, a free object on a set $A$ represents the ''most general" orthomodular lattice that can be generated from the elements of $A$. Remarkably, this free object can be explicitly characterized as a complete powerset Boolean algebra $\powerset{A}$.

\medskip

\begin{proposition}
\label{FreeOMLatGalProp}
A free object on a  set $A$ in $\Cat{SupOMLatLin}$ is isomorphic to the complete powerset Boolean algebra $\powerset{A}$.
\end{proposition}
\begin{proof}  The canonical injection 
$i\colon A \to \powerset{A}$ is given by 
$a\mapsto \{a\}$. Let $Y$ be a complete orthomodular lattice 
and $g\colon A\to Y$. We define 
$f\colon \powerset{A}\to Y$ by 
$f(Z)=\bigvee \{g(a)\mid a\in Z\}$ and 
$f^{*}\colon Y\to \powerset{A}$ by 
$f^{*}(y)=\{a\in A\mid g(a) \not\perp y\}$. 
Clearly, $f\after i = g$.

We compute: 
\begin{align*}
f(Z)\perp y\ &
   \text{ if and only if } g(a)\perp y\  
   \text{ for all } a\in Z \\
   &\text{ if and only if } a\not\in f^{*}(y) 
    \text{ for all } a\in Z \\
     &\text{ if and only if } f^{*}(y)\subseteq A\setminus Z\\
     &\text{ if and only if } Z\perp f^{*}(y).
\end{align*}

\noindent Moreover,
if $h\colon  \powerset{A} \rightarrow Y$ also satisfies
$h\after i = g$, then: 
\begin{align*}
h(Z)&=
\bigvee_{a\in Z} h(\{a\})=%
\bigvee_{a\in Z} g(a)=f(Z). 
\end{align*}
\end{proof}

\section{Conclusion}\label{ConclusionSec}

Our research demonstrates that, akin to the findings in \cite{Jac}, $\Cat{OMLatLin}$ constitutes a dagger kernel category. In a forthcoming paper, we plan to associate with each complete orthomodular lattice $X$ a quantale $\Cat{Lin}(X)$ comprising the endomorphisms of $X$, thereby introducing a natural fuzzy-theoretic perspective to the theory of complete orthomodular lattices. Our analysis has revealed that $\Cat{SupOMLatLin}$ exhibits the behavior of a quantaloid as introduced in \cite{rosenthal2}. We intend to demonstrate that $\Cat{SupOMLatLin}$ is indeed an involutive quantaloid.

\backmatter

\bmhead{Acknowledgements}

The first author acknowledges the support of project 23-09731L by the Czech Science Foundation (GA\v CR), entitled 
``Representations of algebraic semantics for substructural logics''.
The research of the last two authors was supported by the Austrian Science Fund (FWF), {[10.55776/PIN5424624]} and the Czech Science Foundation (GA\v CR) project 24-15100L, entitled ``Orthogonality and Symmetry''. The third author was supported by the project MUNI/A/1457/2023 by Masaryk University.

\bibliography{sn-bibliography}

\input output.bbl

\end{document}